\documentclass[12pt]{article}
\usepackage{amssymb}
\usepackage{amsmath}
\usepackage{amsthm}
\usepackage{color}
\usepackage{comment}
\usepackage{geometry}		% allows easy margin settings
\usepackage{graphicx}

\let\originalleft\left
\let\originalright\right
\renewcommand{\left}{\mathopen{}\mathclose\bgroup\originalleft}
\renewcommand{\right}{\aftergroup\egroup\originalright}

\begin{document}

\def\cB{\mathcal{B}}
\def\cL{\mathcal{L}}
\def\cM{\mathcal{M}}
\def\cN{\mathcal{N}}
\def\cO{\mathcal{O}}
\def\cW{\mathcal{W}}
\def\cX{\mathcal{X}}
\def\cY{\mathcal{Y}}
\def\rD{{\rm D}}
\def\ee{\varepsilon}
\def\halfOfACommaSpace{\hspace{.3mm}}
\def\manualEndProof{\hfill $\Box$}

\newcommand{\removableFootnote}[1]{}

\newtheorem{theorem}{Theorem}[section]
\newtheorem{corollary}[theorem]{Corollary}
\newtheorem{lemma}[theorem]{Lemma}
\newtheorem{proposition}[theorem]{Proposition}

\theoremstyle{definition}
\newtheorem{definition}{Definition}[section]
\newtheorem{assumption}[definition]{Assumption}
\newtheorem{example}[definition]{Example}

\theoremstyle{remark}
\newtheorem{remark}{Remark}[section]

% To do:
%   - allow a<0?

% Style comments:
%		? spelling: GB English
%		? uses fully standardised figure sizes, hh scaled by 4/3
%		-	max figure width used for IJBC (World Scientific), two columns, 8.0cm & 17.0cm
%		- abbreviations:

\title{
Unfolding codimension-two subsumed homoclinic connections in two-dimensional piecewise-linear maps.
}
\author{
D.J.W.~Simpson\\\\
School of Fundamental Sciences\\
Massey University\\
Palmerston North\\
New Zealand
}
\maketitle

% keywords:
% 	homoclinic tangency; homoclinic corner; piecewise-smooth; mode-locking; multi-stability

% MSC codes:
% 	37G25 -- Bifurcations connected with nontransversal intersection
%		39A23 -- Periodic solutions
%		39A28 -- Bifurcation theory

\begin{abstract}

For piecewise-linear maps, the phenomenon that a branch of a one-dimensional unstable manifold of a periodic solution is completely contained in its stable manifold is codimension-two. Unlike codimension-one homoclinic corners, such `subsumed' homoclinic connections can be associated with stable periodic solutions. The purpose of this paper is to determine the dynamics near a generic subsumed homoclinic connection in two dimensions. Assuming the eigenvalues associated with the periodic solution satisfy $0 < |\lambda| < 1 < \sigma < \frac{1}{|\lambda|}$, in a two-parameter unfolding there exists an infinite sequence of roughly triangular regions within which the map has a stable single-round periodic solution. The result applies to both discontinuous and continuous maps, although these cases admit different characterisations for the border-collision bifurcations that correspond to boundaries of the regions. The result is illustrated with a discontinuous map of Mira and the two-dimensional border-collision normal form.

\end{abstract}

%===============================================================================
\section{Introduction}
\label{sec:intro}
\setcounter{equation}{0}

%We first describe homoclinic tangencies. % in their simplest setting.
%Let $\Gamma$ be a saddle-type period-$n$ solution of a diffeomorphism on $\mathbb{R}^2$.
We begin by briefly reviewing related classical results for smooth maps.
Let $\Gamma$ be a period-$n$ solution of a diffeomorphism on $\mathbb{R}^2$
with eigenvalues satisfying $0 < |\lambda| < 1 < |\sigma|$.
Then $\Gamma$ has one-dimensional stable and unstable manifolds;
a tangential intersection between these is a {\em homoclinic tangency}.
Generically the tangency is quadratic, to leading order.
Homoclinic tangencies are global, codimension-one bifurcations that create chaos
or represent a crisis where a chaotic attractor suddenly changes size \cite{PaTa93}.

%In the simplest setting, a periodic solution, call it $\Gamma$, of a diffeomorphism on $\mathbb{R}^2$ has
%The bifurcation structure near a homoclinic tangency was originally investigated by
%Gavrilov and Shilnikov \cite{GaSi72,GaSi73}.
%Let $\Gamma$ be a periodic solution of a diffeomorphism on $\mathbb{R}^2$ with
%associated eigenvalues $\lambda$ and $\sigma$ that satisfy $0 < \lambda < 1 < \sigma$.
%Then $\Gamma$ has one-dimensional stable and unstable manifolds
%and a tangential intersection between these is what is meant by a homoclinic tangency.

%Homoclinic tangencies are codimension-one phenomena,
%so let us consider how the dynamics changes under the variation of one scalar parameter of the map.
%Suppose we vary a parameter $\xi \in \mathbb{R}$ of the map and a homoclinic tangency occurs at $\xi = \xi^*$.
Now suppose $\xi \in \mathbb{R}$ is a parameter of the map and $\Gamma$ has a homoclinic tangency at $\xi = \xi^*$.
%If the absolute value of the product of the eigenvalues associated with $\Gamma$ is less than $1$,
%If $\lambda \sigma < 1$ at $\xi = \xi^*$,
Typically there exist {\em single-round} periodic solutions for values of $\xi$ near $\xi^*$.
These have periods $k n + p$ for fixed $p$ and different values of $k$.
%A single-round periodic solution consists of $k$ points 'near' $\Gamma$ and one excursion far from $\Gamma$,
%call it a $k$-single-round periodic solution.
If $\xi$ unfolds the tangency in a generic fashion and $|\lambda \sigma| < 1$ at $\xi = \xi^*$,
then there exists an infinite sequence of intervals $I_k$
such that the map has a stable period-$(k n + p)$ solution for all $\xi \in I_k$.
%then there exists an infinite sequence of intervals of $\xi$-values
%within which the map has a stable period-$(k n + p)$ solution.
%then the map has a stable period-$(k n + p)$ solution
%there exists a stable $k$-single-round periodic solution
%in an interval of $\xi$-values for all sufficiently large values of $k$.
As $k \to \infty$ the intervals converge to $\xi^*$ and are non-overlapping \cite{GaSi72}.

In two-dimensional parameter space homoclinic tangencies occur on curves.
These curves may contain points where the tangency is codimension-two,
such as the intersection between two homoclinic tangency curves associated with the same periodic solution.
%These curves may intersect or in other ways have points at which a tangency is codimension-two.
A determination of the dynamics near a codimension-two point %--- an {\em unfolding} ---
is often helpful as it explains other curves of codimension-one bifurcations that
%as it describes the dynamics in a neighbourhood
%of the point and often explains the origin of other codimension-one bifurcation curves
may extend over large areas of parameter space \cite{Da91,GoTu07,HiLa95}.

This paper concerns piecewise-linear maps on a domain $\Omega \subset \mathbb{R}^2$ that we write as
\begin{equation}
f(x,y) =
\begin{cases}
f_1(x,y), & (x,y) \in \Omega_1 \,, \\[-1.6mm]
\hspace{9mm} \vdots \\[-2.3mm]
f_m(x,y), & (x,y) \in \Omega_m \,.
\end{cases}
\label{eq:f}
\end{equation}
The regions $\Omega_1,\ldots,\Omega_m$ partition the domain $\Omega$,
and each $f_j$ is assumed to be affine (linear plus a constant).
We let $\Sigma$ denote the set of all $P \in \Omega$ for which every neighbourhood of $P$ intersects more than one $\Omega_j$.
We assume $\Sigma$ is a finite union of smooth curves;
it is the collection of {\em switching manifolds} of \eqref{eq:f}.
%We do not require \eqref{eq:f} to be continuous at points in $\Sigma$.
%across the boundaries of the $\Omega_j$.

%%%%%%%%%%%%%%%%%%%%%%%%%%%%%%%%%%%%%%%%%%%%%%%%%%%%%%%%%%%%%
\begin{figure}[b!]
\begin{center}
\setlength{\unitlength}{1cm}
\begin{picture}(5.2,3.9)
\put(0,0){\includegraphics[height=3.9cm]{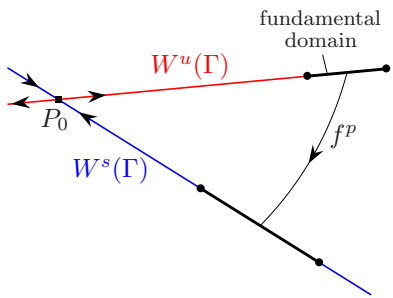}}
\put(.44,2.24){\footnotesize $P_0$}
\put(4.28,2.02){\footnotesize $f^p$}
\put(3.35,3.59){\scriptsize fundamental}
\put(3.69,3.3){\scriptsize domain}
\put(1.9,2.95){\footnotesize \color{red} $W^u(\Gamma)$}
\put(.86,1.57){\footnotesize \color{blue} $W^s(\Gamma)$}
\end{picture}
\caption{
A sketch of the stable (blue) and unstable (red) manifolds of a saddle-type
periodic solution $\Gamma$ of a piecewise-linear map $f$
in a neighbourhood of one point of this solution, $P_0$.
In this neighbourhood the manifolds coincide with the
stable and unstable subspaces associated with $P_0$.
In the simplest instance of a codimension-two subsumed homoclinic connection,
a fundamental domain of a branch of the unstable manifold
%(the domain contains exactly one point of every orbit in the branch)
maps linearly to the stable subspace under $p$ iterations of $f$, for some $p \ge 1$.
\label{fig:codimTwoHCCornerSchem_a}
}
\end{center}
\end{figure}
%%%%%%%%%%%%%%%%%%%%%%%%%%%%%%%%%%%%%%%%%%%%%%%%%%%%%%%%%%%%%

%Now let $\Gamma$ be
Now consider a period-$n$ solution of \eqref{eq:f} that has no points in $\Sigma$
and eigenvalues satisfying $0 < |\lambda| < 1 < |\sigma|$.
Its stable and unstable manifolds are piecewise-linear, so cannot form a quadratic tangency.
Instead they may form a {\em homoclinic corner}: locally one manifold is linear while the other has a kink.
%At the analogous bifurcation where stable and unstable manifolds first attain an intersection,
%at each point of intersection one manifold has a kink, or corner.
But unlike a homoclinic tangency, even if $|\lambda \sigma| < 1$ close to a generic homoclinic corner
there are no stable periodic solutions \cite{Si16b}.
%This is because absence of a smooth turning point means there is no zero derivative to generate stability.

At a homoclinic corner, one orbit in the unstable manifold belongs to the stable manifold.
It is a codimension-two phenomenon for two different orbits to have this property.
%This paper concerns the codimension-two scenario
%that two points of a fundamental domain of a branch of the unstable manifold intersect the stable manifold.
But since \eqref{eq:f} is piecewise-linear, this can imply {\em all} orbits
in a branch of the unstable manifold belong to the stable manifold,
%(and this is an assumption we include below; it does not add to the codimension of the phenomenon),
Fig.~\ref{fig:codimTwoHCCornerSchem_a}.
In this case the branch is completely contained within the stable manifold and we say it is {\em subsumed}.

%A {\em fundamental domain} of a one-dimensional unstable manifold
%is a section of the manifold that contains exactly one point of every orbit in the manifold.
%The unstable manifold can be `grown' by iterating a fundamental domain,
%indeed this is how one usually computes unstable manifolds numerically \cite{Ho93,YoKo91}.
%Assuming a saddle-type periodic solution of \eqref{eq:f} involves no points on the boundaries of the $\Omega_j$,
%by looking near the periodic solution we can always find a fundamental domain of its unstable manifold that is simply a line segment.
%At a homoclinic corner, one point in the fundamental domain maps to the stable subspace under, say, $p$ iterations of $f$.
%This paper concerns the codimension-two scenario that two points in the fundamental domain map to the stable subspace under $p$ iterations of $f$.
%Since $f$ is piecewise-linear we may have that all points between these ...
%In two-dimensional parameter space this generically corresponds to the intersection of two curves of homoclinic corners.

In two-dimensional parameter space subsumed homoclinic connections
typically lie at the intersection of two curves of homoclinic corners.
%The intersection of two curves of homoclinic corners was been identified numerically
This occurs for a discontinuous map introduced by Mira in \cite{Mi13},
and the two-dimensional border-collision normal \cite{Si16b}.
In both cases (given as examples below) the intersection is a striking focal point
of the overall bifurcation structure. % associated with attracting invariant sets.
This is because these examples involve $|\lambda \sigma| < 1$, and,
as explained below by Theorem \ref{th:main},
in this case there exists $p \ge 1$ and a sequence of regions $S_k$
in which the map has a stable period-$(k n + p)$ solution
for all sufficiently large values of $k$.
Each $S_k$ is triangular, to leading order,
and overlaps only $S_{k-1}$ and $S_{k+1}$ for large $k$.
%As $k \to \infty$ the regions approach the intersection point
%with $S_k$ only overlapping $S_{k+1}$ and $S_{k-1}$.

The remainder of the paper is organised as follows.
In \S\ref{sec:periodic} we characterise periodic solutions of \eqref{eq:f}
and discuss their existence, admissibility, and stability.
Then in \S\ref{sec:thm} we carefully set up the codimension-two scenario, assuming $\sigma > 1$ for simplicity,
%which allows us to make precise quantitative statements, specifically
and state Theorem \ref{th:main}.
In \S\ref{sec:proof} we provide a proof of Theorem \ref{th:main} and
%Theorem \ref{th:main} is proved in \S\ref{sec:proof}.
%The proof is achieved by finding three bifurcation curves where a period-$(k n + p)$ solution
%undergoes a bifurcation, and showing that in the region between these curves the solution is admissible and stable.
in \S\ref{sec:examples} illustrate the result with examples.
Finally \S\ref{sec:conc} provides concluding remarks. %a brief discussion and outlook for future work. 

%===============================================================================
\section{Periodic solutions: notation and basic properties}
\label{sec:periodic}
\setcounter{equation}{0}

We first explain how periodic solutions of \eqref{eq:f} can be described symbolically.
For example, the period-$3$ solution shown in Fig.~\ref{fig:codimTwoHCCornerSchem_b} is a $135$-cycle
(cyclic permutations of $135$ are also permitted).
%if we take the point in $\Omega_1$ as the `first' point.
Ostensibly, $135$ is a number, but is termed a `word'
because the indices $1,\ldots,m$ are treated as symbols \cite{HaZh98}.
A similar exposition involving only two symbols is presented in \cite{Si17c}.

%%%%%%%%%%%%%%%%%%%%%%%%%%%%%%%%%%%%%%%%%%%%%%%%%%%%%%%%%%%%%
\begin{figure}[t!]
\begin{center}
\setlength{\unitlength}{1cm}
\begin{picture}(8,6)
\put(0,0){\includegraphics[height=6cm]{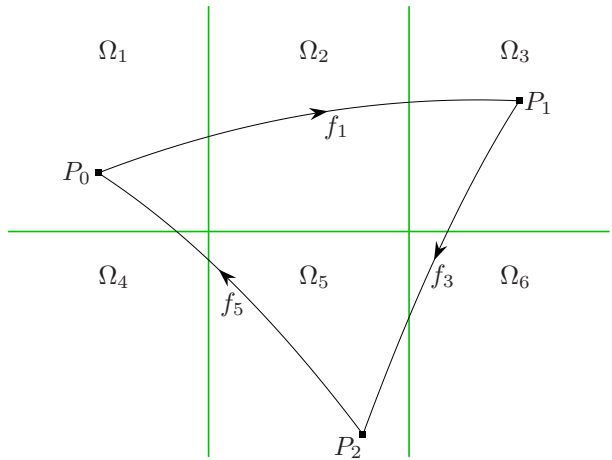}}
\put(1.20,5.3){\footnotesize $\Omega_1$}
\put(3.87,5.3){\footnotesize $\Omega_2$}
\put(6.54,5.3){\footnotesize $\Omega_3$}
\put(1.20,2.3){\footnotesize $\Omega_4$}
\put(3.87,2.3){\footnotesize $\Omega_5$}
\put(6.54,2.3){\footnotesize $\Omega_6$}
\put(.72,3.68){\footnotesize $P_0$}
\put(6.86,4.62){\footnotesize $P_1$}
\put(4.33,.03){\footnotesize $P_2$}
\put(4.2,4.3){\footnotesize $f_1$}
\put(5.61,2.31){\footnotesize $f_3$}
\put(2.82,1.94){\footnotesize $f_5$}
\end{picture}
\caption{
A sketch of the phase space of a map \eqref{eq:f} that involves six regions
showing a $135$-cycle, $\{ P_0, P_1, P_2 \}$.
The green lines represent $\Sigma$.
\label{fig:codimTwoHCCornerSchem_b}
}
\end{center}
\end{figure}
%%%%%%%%%%%%%%%%%%%%%%%%%%%%%%%%%%%%%%%%%%%%%%%%%%%%%%%%%%%%%

We write a {\em word} $\cX$ of length $n$ as
\begin{equation}
\cX = \cX_0 \cdots \cX_{n-1} \,,
\label{eq:X}
\end{equation}
where $\cX_i \in \{ 1,\ldots,m \}$ for each $i \in \{ 0,\ldots,n-1 \}$.
We write $\cX \cY$ for the concatenation of two words $\cX$ and $\cY$,
and $\cX^k$ for the concatenation of $\cX$ with itself $k$ times.

Given $i \ge 0$, let $f^i$ denote the composition of $f$ with itself $i$ times
($f^0$ is the identity map).
Given a word $\cX$ of length $n$, let
\begin{equation}
f_\cX = f_{\cX_{n-1}} \circ \cdots \circ f_{\cX_0} \,,
\label{eq:fX}
\end{equation}
denote the composition of the pieces of $f$ in the order determined by $\cX$.
%(alternatively we can generate
%a word by iterating a point under $f$, then $\Omega_\cX$
%is all points that generate $\cX$)
%\begin{equation}
%\Omega_\cX = \left\{ P \in \mathbb{R}^2 \,\middle|\,
%P \in \Omega_{\cX_0} \,,
%f_{\cX_0} \in \Omega_{\cX_1} \,, \ldots,
%f_{\cX_0 \cdots \cX_{n-2}} \in \Omega_{\cX_{n-1}} \right\}.
%\label{eq:OmegaX}
%\end{equation}
Let $\Omega_\cX \subset \Omega$ be the (possibly empty)
set of all $P \in \Omega$ for which $f^i(P) \in \Omega_{\cX_i}$
for all $i \in \{ 0,\ldots,n-1 \}$.
Roughly speaking, $\Omega_\cX$ is the set of all points whose forward orbits follow $\cX$ initially.
Notice $f^n(P) = f_\cX(P)$ for all $P \in \Omega_\cX$.

%The boundary of $\Omega_\cX$ is contained in $\Sigma_n$, where we define
%\begin{equation}
%\Sigma_n = \Sigma \cup f^{-1}(\Sigma) \cup \cdots \cup f^{-(n-1)}(\Sigma).
%\nonumber
%\end{equation}

%...............................................................................
\begin{definition}
Given a word $\cX$ of length $n$,
an {\em $\cX$-cycle} is an $n$-tuple,
$\left\{ P_0, P_1, \ldots, P_{n-1} \right\}$, for which
$f_{\cX_0}(P_0) = P_1,
f_{\cX_1}(P_1) = P_2, \ldots,
f_{\cX_{n-1}}(P_{n-1}) = P_0$.
\label{df:Xcycle}
\end{definition}

Notice $P_0$ in Definition \ref{df:Xcycle} is a fixed point of $f_\cX$.
Since $f_\cX$ is affine, 
its Jacobian matrix $\rD f_\cX$ is constant in phase space
and we have the following result.

%...............................................................................
\begin{lemma}
A map \eqref{eq:f} has a unique $\cX$-cycle if and only if $I - \rD f_\cX$ is non-singular.
\label{le:existence}
\end{lemma}

%...............................................................................
\begin{remark}
Lemma \ref{le:existence} implies
the $\cX$-cycle is completely characterised by the word $\cX$,
except in the special case that $I - \rD f_\cX$ is singular
(which is equivalent to $1$ being an eigenvalue of $\rD f_\cX$).
\end{remark}

%...............................................................................
\begin{definition}
The $\cX$-cycle of Definition \ref{df:Xcycle} is
{\em admissible} if $P_0 \in \Omega_\cX$,
and {\em virtual} otherwise.
\end{definition}

%...............................................................................
\begin{remark}
Every admissible $\cX$-cycle is a periodic solution of $f$.
Conversely every periodic solution of $f$ is an admissible $\cX$-cycle
for some $\cX$ that is unique up to cyclic permutation.
\end{remark}

%Note that Lemma \ref{le:existence} says nothing about admissibility
%which for continuous maps with two pieces can be characterised
%by the determinants of a set of $n$ matrices \cite{SiMe10}.

%This is equivalent to $1$ not being an eigenvalue of $\rD f_\cX$.
%Notice
%\begin{equation}
%\rD f_\cX = A_{\cX_{n-1}} \cdots A_{\cX_0} \,.
%\label{eq:DfX}
%\end{equation}

If an admissible $\cX$-cycle has no points in $\Sigma$ (as in Fig.~\ref{fig:codimTwoHCCornerSchem_b})
then some neighbourhood of $P_0$ is contained in $\Omega_\cX$.
In this neighbourhood $f^n$ is smooth (in fact equal to $f_\cX$ and so affine).
Thus, as with smooth maps \cite{El08},
the $\cX$-cycle is asymptotically stable if both eigenvalues of $\rD f_\cX$ have modulus less than $1$.
This corresponds to shaded region shown in Fig.~\ref{fig:codimTwoHCCornerSchem_c}
and in fact the converse is true because $f_\cX$ is affine:

%...............................................................................
\begin{lemma}
Suppose an admissible $\cX$-cycle has no points in $\Sigma$.
Let $\tau$ and $\delta$ be the trace and determinant of $\rD f_\cX$.
Then the $\cX$-cycle is asymptotically stable if and only if $(\tau,\delta)$ belongs
to the interior of the triangle shown in Fig.~\ref{fig:codimTwoHCCornerSchem_c}.
\label{le:stability}
\end{lemma}

%%%%%%%%%%%%%%%%%%%%%%%%%%%%%%%%%%%%%%%%%%%%%%%%%%%%%%%%%%%%%
\begin{figure}[t!]
\begin{center}
\setlength{\unitlength}{1cm}
\begin{picture}(5.2,3.9)
\put(0,0){\includegraphics[height=3.9cm]{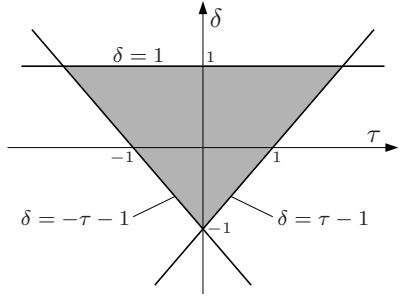}}
\put(4.78,2.02){\footnotesize $\tau$}
\put(2.69,3.55){\footnotesize $\delta$}
\put(.18,.92){\scriptsize $\delta = -\tau - 1$}
\put(3.59,.92){\scriptsize $\delta = \tau - 1$}
\put(1.4,3.09){\scriptsize $\delta = 1$}
\put(1.36,1.75){\tiny $-1$}
\put(3.52,1.75){\tiny $1$}
\put(2.65,.81){\tiny $-1$}
\put(2.64,3.08){\tiny $1$}
\end{picture}
\caption{
The shaded region is where both eigenvalues of a $2 \times 2$ matrix,
with trace $\tau$ and determinant $\delta$, have modulus less than $1$.
\label{fig:codimTwoHCCornerSchem_c}
}
\end{center}
\end{figure}
%%%%%%%%%%%%%%%%%%%%%%%%%%%%%%%%%%%%%%%%%%%%%%%%%%%%%%%%%%%%%

%...............................................................................
%\begin{remark}
%The following statements are equivalent:
%(i) $I - \rD f_\cX$ is singular (see Lemma \ref{le:existence}),
%(ii) $\delta = \tau - 1$ (see Lemma \ref{le:stability}),
%(iii) $1$ is an eigenvalue of $\rD f_\cX$.
%\end{remark}

%The stability of periodic solutions with points on $\Sigma$
%is significantly more complicated \cite{Si16d}.

%===============================================================================
\section{Subsumed homoclinic connections and single-round periodic solutions}
\label{sec:thm}
\setcounter{equation}{0}

In this section we develop a precise symbolic and geometric description of
subsumed homoclinic connections building up to the statement of Theorem \ref{th:main}.
Theorem \ref{th:main} is essence quite simple.
It gives, explicitly but just to leading order, the set of all points where the map has a stable single-round periodic solution.
However, some technicalities precede the theorem statement because in order for us to provide an explicit formula we need to
construct suitable coordinates in parameter space
and to prove that single-round periodic solutions are admissible we require certain
global assumptions on the nature of the homoclinic connection.

%-------------------------------------------------------------------------------
%\subsection{Subsumed homoclinic connections}
%\subsection{Stable and unstable manifolds}
\subsection{Persistent properties of the stable and unstable manifolds of an $\cX$-cycle}

Let $\Gamma = \left\{ P_0, P_1, \ldots, P_{n-1} \right\}$ be an admissible $\cX$-cycle, as in Definition \ref{df:Xcycle}.
The {\em stable manifold} of the $\cX$-cycle, $W^s(\Gamma)$,
is defined as all $P \in \Omega \setminus \Gamma$
for which $f^i(P)$ converges to the $\cX$-cycle as $i \to \infty$.
The {\em unstable manifold} of the $\cX$-cycle, $W^u(\Gamma)$,
is defined as all $P \in \Omega \setminus \Gamma$
for which there exists a sequence of preimages $f^i(P)$
that converges to the $\cX$-cycle as $i \to -\infty$.

Suppose the $\cX$-cycle has no points in $\Sigma$
and the eigenvalues of $\rD f_\cX$, call them $a$ and $b$, satisfy $0 < |a| < 1 < b$.
Then $W^s(\Gamma)$ and $W^u(\Gamma)$
are one-dimensional and emanate linearly from each point $P_i$.
As they emanate from $P_0$, for example, they coincide with
the stable and unstable subspaces $E^s(P_0)$ and $E^u(P_0)$.
These are lines through $P_0$ with directions given by the eigenvectors of $\rD f_\cX$.
%Let $W^s_{\rm loc}(P_0)$ and $W^u_{\rm loc}(P_0)$ be the parts of the stable and unstable manifolds
%(as they emanate linearly from $P_0$) that coincide with $E^s(P_0)$ and $E^u(P_0)$.
%Note that $W^s_{\rm loc}(P_0)$ and $W^u_{\rm loc}(P_0)$
Note that $W^u(\Gamma)$ has two {\em dynamically independent} branches because $b > 0$.

%We now construct a fundamental domain for one branch of the unstable manifold.
%Usually this domain will be line segment with one endpoint
%where the branch first differs from $E^u(P_0)$,
%and the other endpoint its preimage under $f_\cX$.
We introduce the notation
\begin{equation}
\cL(P,Q) = \left\{ (1-t) P + t Q \,\middle|\, 0 < t < 1 \right\},
\nonumber
\end{equation}
to denote all points between $P, Q \in \mathbb{R}^2$ (if $P \ne Q$ this is a line segment
that does not include its endpoints).
Observe $P_0$ belongs to the interior of $\Omega_\cX$.
Suppose $E^u(P_0)$ intersects $\Sigma$ at some $Z \in \Omega$, and $\cL(P_0,Z) \subset \Omega_\cX$,
see already Fig.~\ref{fig:codimTwoHCCornerSchem_d}.
%As we follow one branch of the unstable manifold outwards from $P_0$,
%let $Z$ be the first point we reach that belongs to $\partial \Omega_\cX$
%(the boundary of $\Omega_\cX$).
%see already Fig.~\ref{fig:codimTwoHCCornerSchem_d}.
%Between $P_0$ and $f_\cX(Z)$ the branch coincides with $E^u(P_0)$
%because preimages under $f_\cX$ belong to $\Omega_\cX$.
%Then the set consisting of $Z$ and all points between $Z$ and $f_\cX(Z)$
%is a fundamental domain for this branch
%and, except in special cases, beyond $f_\cX(Z)$ the branch does not coincide with $E^u(P_0)$ 
%because $Z \in \partial \Omega_\cX$.
%Given $P, Q \in \mathbb{R}^2$, let
Then every $P \in \cL \left( P_0, f_\cX(Z) \right)$ belongs to one branch of $W^u(\Gamma)$
because the preimage of $P$ under $f_\cX$ belongs to $\Omega_\cX$.
For this reason $\{ Z \} \cup \cL \left( Z, f_\cX(Z) \right)$ is a {\em fundamental domain} for
this branch of $W^u(\Gamma)$ (it contains exactly one point in every orbit in the branch).
This particular fundamental domain is significant because $Z \in \Sigma$
and thus $W^u(\Gamma)$ has a kink at $f_\cX(Z)$ (except in special cases).

%be the stable and unstable subspaces of $P_0$ for the map $f_\cX$, which we denote $E^s(P_0)$ and $E^u(P_0)$
%(these are lines with directions given by the eigenvectors of $\rD f_\cX$).

%Now let us look just about the point $P_0$.
%Let $E^s(P_0)$ and $E^u(P_0)$ be the stable and unstable subspaces of $P_0$ for the map $f_\cX$
%(these are lines with directions given by the eigenvectors of $\rD f_\cX$).
%As the stable and unstable manifolds of the $\cX$-cycle emanate from $P_0$,
%they coincide with $E^s(P_0)$ and $E^u(P_0)$.

%For our purposes it suffices to consider these manifolds more carefully near one point, say $P_0$.
%For the affine map $f_\cX$, let $E^s(P_0)$ and $E^u(P_0)$ be the stable and unstable manifolds of its fixed point $P_0$.

%Since the $\cX$-cycle has no points in $\Sigma$, the set $\Omega_\cX$ contains $P_0$ in its interior.
%Thus there exists a convex set $\cB \subset \Omega_\cX$ that contains $P_0$ in its interior.
%Let $E^s(P_0)$ and $E^u(P_0)$ be the stable and unstable subspaces of $P_0$ for the map $f_\cX$
%(these are lines with directions given by the eigenvectors of $\rD f_\cX$).
%Since $f_\cX$ is affine, in $\cB$ the parts of the stable and unstable manifolds of the $\cX$-cycle
%that emanate from 
%For diffeomorphisms the stable manifold theorem tells us that the stable and unstable manifolds
%are tangent to the stable and unstable 

%-------------------------------------------------------------------------------
%\subsection{Assumptions at the codimension-two point}
%\subsection{Subsumed homoclinic connections and single-round periodic solutions}
\subsection{Subsumed homoclinic connections}

%...............................................................................
\begin{definition}
The $\cX$-cycle is said to have a {\em subsumed homoclinic connection}
if one branch of $W^u(\Gamma)$ is a subset of $W^s(\Gamma)$
(or one branch of $W^s(\Gamma)$ is a subset of $W^u(\Gamma)$).
\label{df:shc}
\end{definition}

%In order to establish the existence and admissibility of single-round periodic solutions
%Here we describe a subsumed homoclinic connection
%subject to additional symbolic constraints (which occur typically)\removableFootnote{
%In short we assume that all homoclinic orbits have the symbol sequence
%$\cX^\infty \cY \cX^\infty$, for a fixed word $\cY$, like \cite{SiTu17}.
%}.
%In Theorem \ref{th:main} we make some precise assumptions on the symbolic nature
%of the homoclinic connection in order to prove the existence of stable single-round periodic solutions.

Here we provide sufficient conditions for an $\cX$-cycle to have a subsumed homoclinic connection.
Let $\cY$ be a word of length $p$.
Let $\cB$ be a compact, convex set with $P_0 \in \cB \subset {\rm int} \left( \Omega_{\cX^2} \right)$
(where ${\rm int}(\cdot)$ denotes the interior of a set).
In Theorem \ref{th:main} we unfold the codimension-two scenario that $Z$ and $f_\cX(Z)$ both
map to $\cB \cap E^s(P_0)$ under $f_\cY$, that is
\begin{equation}
f_\cY(Z), f_{\cX \cY}(Z) \in \cB \cap E^s(P_0).
\label{eq:codimTwoCondition}
\end{equation}
Since $f_\cY$ is affine, \eqref{eq:codimTwoCondition} implies all points
in $\cL \left( Z, f_\cX(Z) \right)$ map to $\cB \cap E^s(P_0)$ under $f_\cY$.
The assumption that $\cB \subset {\rm int} \left( \Omega_{\cX^2} \right)$ is convex ensures
$\cB \cap E^s(P_0)$ is contained in $W^s(\Gamma)$
(if $a > 0$ the weaker assumption $\cB \subset {\rm int} \left( \Omega_\cX \right)$ is sufficient).
For \eqref{eq:codimTwoCondition} to imply a subsumed homoclinic connection
we need the admissibility assumption
\begin{equation}
\cL \left( Z, f_\cX(Z) \right) \subset \Omega_\cY \,,
\label{eq:basicAdmissibilityAssumption}
\end{equation}
and some information about the forward orbit of $Z$.
%We avoid the assumption $Z \in \Omega_\cY$
%because in practice \eqref{eq:basicAdmissibilityAssumption} can only be satisfied
%if the forward orbit of $Z$ under $f$ intersects $\Sigma$,
%so its evolution depends on which $f_j$ is applied at this point of intersection
%(and this depends on the particular map under consideration and
%has no significance to the single-round periodic solutions).

%...............................................................................
\begin{proposition}
Suppose \eqref{eq:codimTwoCondition} and \eqref{eq:basicAdmissibilityAssumption} are satisfied.
Suppose either (i) $Z \in \Omega_\cY$, or (ii) $Z \in \Omega_\cX$ and $f_\cX(Z) \in \Omega_\cY$.
Then the $\cX$-cycle has a subsumed homoclinic connection.
\label{pr:shc}
\end{proposition}

%...............................................................................
\begin{proof}
For any $P \in \cL \left( Z, f_\cX(Z) \right)$ we have $f^p(P) \in \cB \cap E^s(P_0)$
by \eqref{eq:codimTwoCondition} and \eqref{eq:basicAdmissibilityAssumption} and because $f_\cY$ is affine.
Under condition (i) we have $f^p(Z) \in \cB \cap E^s(P_0)$,
while under condition (ii) we have $f^{n+p}(Z) \in \cB \cap E^s(P_0)$.
In either case every point in the fundamental domain $\{ Z \} \cup \cL \left( Z, f_\cX(Z) \right)$
belongs to $W^s(\Gamma)$, hence $f$ has a subsumed homoclinic connection.
\end{proof}

\subsection{Change of coordinates for variables and parameters}

Now suppose $\Sigma$ and each $f_j$ depend smoothly on parameters $\xi, \zeta \in \mathbb{R}$.
Theorem \ref{th:main} includes the following assumption.

\begin{assumption}
At $(\xi,\zeta) = (0,0)$,
\begin{enumerate}
\setlength{\itemsep}{0pt}
\item
\eqref{eq:f} has an admissible $\cX$-cycle with no points on $\Sigma$
and its eigenvalues satisfy $0 < |a| < 1 < b$,
\item
$Z$, defined as above, lies at a transverse intersection between $E^u(P_0)$ and $\Sigma$, and
\item
$\cL(P_0,Z) \subset {\rm int}(\Omega_{\cX})$ and $\cL \left( Z, f_\cX(Z) \right) \subset {\rm int}(\Omega_\cY)$.
\end{enumerate}
\label{as:main}
\end{assumption}

%Also the $\cX$-cycle, its eigenvalues, and the point $Z$ vary smoothly with respect to $\xi$ and $\zeta$.
Since the $\cX$-cycle is hyperbolic it persists smoothly
as a saddle with no points on $\Sigma$ in a neighbourhood of $(\xi,\zeta) = (0,0)$.
Its eigenvalues $a$ and $b$ vary smoothly with $\xi$ and $\zeta$
as does $Z$ because transverse intersections persist.

Let $q_{\rm unstab} = Z - P_0$.
This is an eigenvector of $\rD f_\cX$ corresponding to the eigenvalue $b$.
Let $q_{\rm stab}$ be a smoothly varying eigenvector of $\rD f_\cX$ corresponding to the eigenvalue $a$.
Form the matrix $T = \begin{bmatrix} q_{\rm stab} & q_{\rm unstab} \end{bmatrix}$.
Then
\begin{equation}
h(P) = T^{-1} (P - P_0)
\label{eq:coordinateChangeVariables}
\end{equation}
gives the coordinates of any $P \in \Omega$ relative to the eigenvectors.
The function $h = (h_1,h_2)$ provides a coordinate change on the variables of \eqref{eq:f}
that we use in \S\ref{sec:proof}.

%%%%%%%%%%%%%%%%%%%%%%%%%%%%%%%%%%%%%%%%%%%%%%%%%%%%%%%%%%%%%
\begin{figure}[t!]
\begin{center}
\setlength{\unitlength}{1cm}
\begin{picture}(12,6)
\put(0,0){\includegraphics[height=6cm]{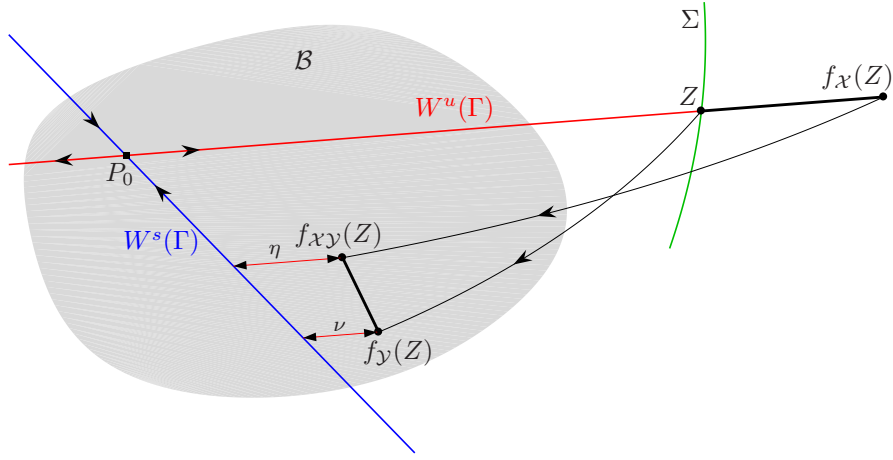}}
\put(1.29,3.61){\footnotesize $P_0$}
\put(8.91,4.63){\footnotesize $Z$}
\put(10.8,4.9){\footnotesize $f_\cX(Z)$}
\put(4.71,1.25){\footnotesize $f_\cY(Z)$}
\put(3.8,2.8){\footnotesize $f_{\cX \cY}(Z)$}
\put(8.94,5.65){\footnotesize $\Sigma$}
\put(3.8,5.1){\footnotesize $\cB$}
\put(3.46,2.64){\scriptsize $\eta$}
\put(4.32,1.62){\scriptsize $\nu$}
\put(5.4,4.48){\footnotesize \color{red} $W^u(\Gamma)$}
\put(1.52,2.71){\footnotesize \color{blue} $W^s(\Gamma)$}
\end{picture}
\caption{
A sketch of the stable (blue) and unstable (red) manifolds of the $\cX$-cycle near $P_0$.
Here the manifolds coincide with the lines $E^s(P_0)$ and $E^u(P_0)$.
The meaning of $\eta$ and $\nu$,
defined by \eqref{eq:eta}--\eqref{eq:nu}, is indicated.
\label{fig:codimTwoHCCornerSchem_d}
}
\end{center}
\end{figure}
%%%%%%%%%%%%%%%%%%%%%%%%%%%%%%%%%%%%%%%%%%%%%%%%%%%%%%%%%%%%%

%Let $\cY$ be a word of length $p$, and
Let
\begin{align}
\eta &= h_2 \left( f_{\cX \cY}(Z) \right),
\label{eq:eta} \\
\nu &= h_2 \left( f_\cY(Z) \right).
\label{eq:nu}
\end{align}
As illustrated in Fig.~\ref{fig:codimTwoHCCornerSchem_d},
these values provide a measure for the distance of $f_{\cX \cY}(Z)$ and $f_\cY(Z)$
from $E^s(P_0)$ in the direction $q_{\rm unstab}$.

Theorem \ref{th:main} assumes the codimension-two condition
\eqref{eq:codimTwoCondition} holds at $(\xi,\zeta) = (0,0)$.
This condition implies $(\eta,\nu) = (0,0)$.
In order for $\xi$ and $\zeta$ to unfold the codimension-two point in a generic fashion
we require $\det(J) \big|_{(\xi,\zeta) = (0,0)} \ne 0$, where
\begin{equation}
J = \begin{bmatrix}
\frac{\partial \eta}{\partial \xi} &
\frac{\partial \eta}{\partial \zeta} \\[1mm]
\frac{\partial \nu}{\partial \xi} &
\frac{\partial \nu}{\partial \zeta}
\end{bmatrix}.
\label{eq:J}
\end{equation}
This ensures that the coordinate change on the parameters,
$(\xi,\zeta) \rightarrow (\eta,\nu)$, is locally invertible,
and we write this coordinate change as
\begin{equation}
(\eta,\nu) = \psi(\xi,\zeta).
\label{eq:coordinateChangeParameters}
\end{equation}
%Theorem \ref{th:main} provides the approximate location of $S_k$ in $(\eta,\nu)$-coordinates.
%The $\eta$-axis is where $Z$ maps to $E^s(P_0)$ under $f_\cY$
%and the $\nu$-axis is where $f_\cX(Z)$ maps to $E^s(P_0)$ under $f_\cY$.

%-------------------------------------------------------------------------------
\subsection{Main result}

An $\cX^k \cY$-cycle, which we write as
$\left\{ \hat{P}^{(k)}_0, \ldots, \hat{P}^{(k)}_{k n + p - 1} \right\}$,
has period $k n + p$ and is a single-round periodic solution in the sense that
it involves a single `excursion' following the word $\cY$.
We centre our analysis around the particular point $\hat{P}^{(k)}_{k n}$
which, roughly speaking, is the point of the $\cX^k \cY$-cycle that gets mapped under $f_\cY$.

%...............................................................................
\begin{definition}
%For each $k \ge 0$, let $S_k \subset \mathbb{R}^2$ be the set of all $(\xi,\zeta)$ for which \eqref{eq:f}
%has an admissible, asymptotically stable $\cX^k \cY$-cycle.
At $(\xi,\zeta) = (0,0)$ let $\cM$ be a neighbourhood of $\cL \left( Z, f_\cX(Z) \right)$.
For each $k \ge 0$, let $S_k \subset \mathbb{R}^2$ be the set of all $(\xi,\zeta)$ for which \eqref{eq:f}
has an admissible, asymptotically stable $\cX^k \cY$-cycle with $\hat{P}^{(k)}_{k n} \in \cM$.
\label{df:Sk}
\end{definition}

In Theorem \ref{th:main}, $\lambda = a(0,0)$ and $\sigma = b(0,0)$ denote the
eigenvalues of the $\cX$-cycle at $(\xi,\zeta) = (0,0)$.
Given $k \ge 0$ and $\gamma \ge 0$, let
\begin{align}
\Delta_{k,\gamma}^\pm = \left\{ (\eta,\nu) \,\middle|\,
\eta < \sigma^{-(k-1)} \pm \gamma^k,
\nu > \sigma^{-k} \mp \gamma^k,
\nu < \eta + (\sigma-1) \sigma^{-k} \pm \gamma^k \right\},
\label{eq:Delta}
\end{align}
shown in Fig.~\ref{fig:codimTwoHCCornerSchem_f}.

%%%%%%%%%%%%%%%%%%%%%%%%%%%%%%%%%%%%%%%%%%%%%%%%%%%%%%%%%%%%%
\begin{figure}[b!]
\begin{center}
\setlength{\unitlength}{1cm}
\begin{picture}(8,6)
\put(0,0){\includegraphics[height=6cm]{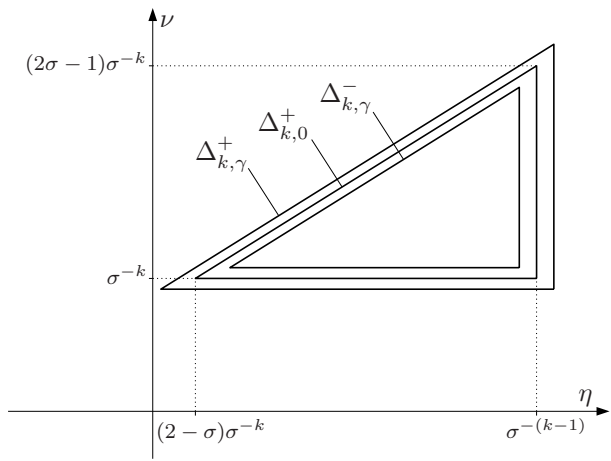}}
\put(7.58,.77){\footnotesize $\eta$}
\put(2.02,5.75){\footnotesize $\nu$}
\put(.23,5.17){\scriptsize $(2 \sigma - 1) \sigma^{-k}$}
\put(1.31,2.33){\scriptsize $\sigma^{-k}$}
\put(1.97,.28){\scriptsize $(2 - \sigma) \sigma^{-k}$}
\put(6.62,.28){\scriptsize $\sigma^{-(k-1)}$}
\put(2.49,3.98){\footnotesize $\Delta_{k,\gamma}^+$}
\put(3.32,4.36){\footnotesize $\Delta_{k,0}^+$}
\put(4.14,4.74){\footnotesize $\Delta_{k,\gamma}^-$}
\end{picture}
\caption{
The regions $\Delta_{k,\gamma}^-$ and $\Delta_{k,\gamma}^+$
provide inner and outer bounds for $\psi(S_k)$,
where $\psi(S_k)$ represents pairs $(\eta,\nu)$
for which \eqref{eq:f} has an admissible, stable $\cX^k \cY$-cycle,
see \eqref{eq:coordinateChangeParameters} and Definition \ref{df:Sk}.
The leading order (large $k$) approximation to $\psi(S_k)$ is $\Delta_{k,0}^+ = \Delta_{k,0}^-$
because we assume $\gamma < \frac{1}{\sigma}$ (Theorem \ref{th:main} provides a specific value for $\gamma$).
\label{fig:codimTwoHCCornerSchem_f}
}
\end{center}
\end{figure}
%%%%%%%%%%%%%%%%%%%%%%%%%%%%%%%%%%%%%%%%%%%%%%%%%%%%%%%%%%%%%

%...............................................................................
\begin{theorem}
Suppose Assumption \ref{as:main} is satisfied,
\eqref{eq:codimTwoCondition} is satisfied at $(\xi,\zeta) = (0,0)$,
and $\cX_0 \ne \cY_0$.
Let $\lambda = a(0,0)$ and $\sigma = b(0,0)$,
suppose $|\lambda| \sigma < 1$, and let
$\gamma = \max \left[ \sigma^{-\frac{3}{2}}, |\lambda|^{\frac{1}{2}} \sigma^{-\frac{1}{2}} \right]$.
Also suppose $\det(J) \ne 0$.
Then there exists $k_{\rm min} \ge 0$,
a neighbourhood $\cM$ as in Definition \ref{df:Sk},
and a neighbourhood $\cN$ of $(\xi,\zeta) = (0,0)$, such that in $\cN$
%$\psi$ is invertible and
\begin{equation}
\psi^{-1} \left( \Delta_{k,\gamma}^- \right) \subset S_k \subset \psi^{-1} \left( \Delta_{k,\gamma}^+ \right),
\label{eq:main}
\end{equation}
for all $k \ge k_{\rm min}$.
\label{th:main}
\end{theorem}

Theorem \ref{th:main}, proved in \S\ref{sec:proof}, provides inner and outer bounds for $S_k$.
Since $\gamma^k \to 0$ faster than $\sigma^{-k} \to 0$,
the relative error that this provides for the location of $S_k$ goes to zero as $k \to \infty$,
specifically $\frac{{\rm Area} \left( \Delta_{k,\gamma}^+ \setminus \Delta_{k,\gamma}^- \right)}
{{\rm Area} \left( \Delta_{k,\gamma}^+ \right)} \to 0$.

We conclude this section with some technical remarks
and determine how the $S_k$ overlap one another, Fig.~\ref{fig:codimTwoHCCornerSchem_e}.
Overlaps are significant as they imply two stable periodic solutions coexist.

%%%%%%%%%%%%%%%%%%%%%%%%%%%%%%%%%%%%%%%%%%%%%%%%%%%%%%%%%%%%%
\begin{figure}[b!]
\begin{center}
\setlength{\unitlength}{1cm}
\begin{picture}(8,6)
\put(0,0){\includegraphics[height=6cm]{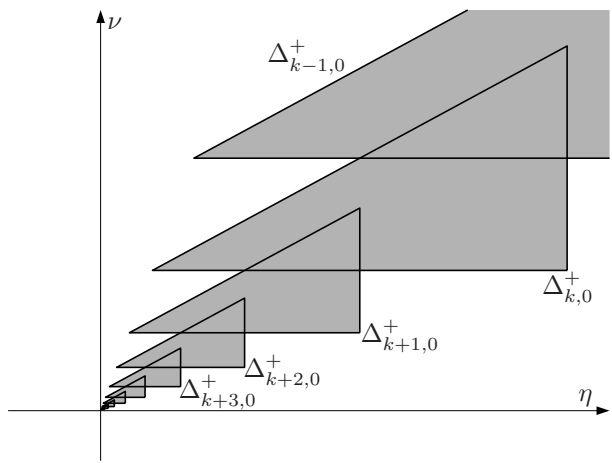}}
\put(7.58,.8){\footnotesize $\eta$}
\put(1.32,5.75){\footnotesize $\nu$}
\put(3.45,5.3){\footnotesize $\Delta_{k-1,0}^+$}
\put(7.1,2.22){\footnotesize $\Delta_{k,0}^+$}
\put(4.69,1.58){\footnotesize $\Delta_{k+1,0}^+$}
\put(3.13,1.11){\footnotesize $\Delta_{k+2,0}^+$}
\put(2.27,.84){\footnotesize $\Delta_{k+3,0}^+$}
\end{picture}
\caption{
A sketch of $\Delta_{k+j,0}^+ = \Delta_{k+j,0}^-$
(the leading order approximation to $\psi(S_{k+j})$)
for several values of $j$.
\label{fig:codimTwoHCCornerSchem_e}
}
\end{center}
\end{figure}
%%%%%%%%%%%%%%%%%%%%%%%%%%%%%%%%%%%%%%%%%%%%%%%%%%%%%%%%%%%%%

%...............................................................................
\begin{remark}
For each of the examples in \S\ref{sec:examples} the $S_k$ involve three smooth boundaries.
To prove that this is the case in general,
which has not been achieved by Theorem \ref{th:main},
we must precisely characterise the bifurcations to which each boundary corresponds.
While the boundary approximated by $\nu = \eta + (\sigma-1) \sigma^{-k}$
is always where the $\cX^k \cY$-cycle loses stability by attaining an eigenvalue of $-1$,
a characterisation of the other boundaries depends on whether $f$ is discontinuous or continuous.
In both cases these boundaries are border-collision bifurcations
where one point of the $\cX^k \cY$-cycle collides with $\Sigma$.
If $f$ is discontinuous, these collisions occur near $Z$.
If $f$ is continuous, $W^u(\Gamma)$ is continuous as so
must have at least one other intersection with $\Sigma$ \cite{SiTu17}.
In this case one of the two collisions occurs near this other intersection
in accordance with the unfolding of a {\em shrinking point} \cite{Si17c,SiMe09}.
\end{remark}

%...............................................................................
\begin{remark}
The assumption $\cX_0 \ne \cY_0$ is a natural consequence of $Z \in \Sigma$ and
allows us to prove $S_k \subset \psi^{-1} \left( \Delta_{k,\gamma}^+ \right)$ in a simple way.
In practise we can always impose $\cX_0 \ne \cY_0$ by choosing an appropriate cyclic permutation of $\cX$.
Similarly the assumption $\hat{P}^{(k)}_{k n} \in \cM$
ensures $S_k$ does not include other points near $(\xi,\zeta) = (0,0)$ due to a peculiar global arrangement
of the regions $\Omega_\cX$ and $\Omega_\cY$.
\end{remark}

%...............................................................................
\begin{remark}
The assumption $\cL \left( Z, f_\cX(Z) \right) \subset {\rm int}(\Omega_\cY)$
is relatively strong as it refers to almost all orbits in the homoclinic connection.
This `global' assumption is used to ensure that the $\cX^k \cY$-cycles are admissible
and was circumvented for a similar situation in \cite{SiTu17}
by having detailed knowledge of one homoclinic orbit and its relation to $\Sigma$.
\end{remark}

%...............................................................................
\begin{remark}
The leading order terms in our
large $k$ asymptotic expansions (see \S\ref{sec:proof})
are $\cO \left( \sigma^{-k} \right)$ (this is big-$O$ notation).
Second order terms are either $\cO \left( \lambda^k \right)$ or $\cO \left( \sigma^{-2 k} \right)$.
We have designed $\gamma$ so that $\gamma^k$ is smaller than leading order
(so that the relative error in $S_k$ decreases with $k$)
but bigger than second order (so that higher order terms can be ignored).
\end{remark}

%...............................................................................
\begin{proposition}
Let $j \ge 1$.
If the conditions of Theorem \ref{th:main} hold then,
for sufficiently large values of $k$,
$S_k \cap S_{k+j} = \varnothing$ if and only if $j > 1$.
\label{pr:overlap}
\end{proposition}

%...............................................................................
\begin{proof}
The top-right corner of $\Delta_{k+1,0}^+$ is
$(\eta,\nu) = \left( \sigma^{-k}, \frac{2 \sigma - 1}{\sigma} \halfOfACommaSpace \sigma^{-k} \right)$.
By checking all three inequalities in \eqref{eq:Delta}
it can easily be seen that this corner belongs to ${\rm int} \left( \Delta_{k,0}^+ \right)$ because $\sigma > 1$.
Thus ${\rm int} \left( \Delta_{k,0}^+ \right) \cap \Delta_{k+1,0}^+ \ne \varnothing$,
and so $S_k \cap S_{k+1} \ne \varnothing$ (using Theorem \ref{th:main})
for sufficiently large values of $k$ because $\gamma^k \to 0$ faster than $\sigma^{-k} \to 0$.

The top-right corner of $\Delta_{k+2,0}^+$ is
$(\eta,\nu) = \left( \sigma^{-(k+1)}, \frac{2 \sigma - 1}{\sigma^2} \halfOfACommaSpace \sigma^{-k} \right)$
and lies below the bottom edge of $\Delta_{k,0}^+$ because
$\frac{2 \sigma - 1}{\sigma^2} = 1 - \frac{(\sigma-1)^2}{\sigma^2} < 1$.
This is also the case for any $\Delta_{k+j,0}^+$ with $j > 2$.
Thus, for all $j > 1$, we have $\Delta_{k,0}^+ \cap \Delta_{k+j,0}^+ = \varnothing$ 
and so $S_k \cap S_{k+j} \ne \varnothing$ (using Theorem \ref{th:main})
for sufficiently large values of $k$.
\end{proof}

%===============================================================================
\section{Proof of Theorem \ref{th:main}}
\label{sec:proof}
\setcounter{equation}{0}

%As discussed immediately following Assumption \ref{as:main},
The $\cX$-cycle exists, is unique, and has no points on $\Sigma$
in some compact neighbourhood $\cN$ of $(\xi,\zeta) = (0,0)$.
We can also assume that in $\cN$ the eigenvalues satisfy $0 < |a| < 1 < b$,
that $Z$ is well-defined and varies smoothly,
and that $h$ and $\psi$ are well-defined and invertible.
%We first express \eqref{eq:f} in terms of the coordinates \eqref{eq:coordinateChangeVariables}.
%In these coordinates $f_\cX$ becomes multiplication by a diagonal matrix, as does $f_{\cX^k}$.
%We then study the existence, admissibility, and stability of $\cX^k \cY$-cycles.

%-------------------------------------------------------------------------------
\subsection{Change of variables}

We write the change of variables \eqref{eq:coordinateChangeVariables} as $(u,v) = h(x,y)$.
For any word $\cW$, let $g_\cW = h \circ f_\cW \circ h^{-1}$
denote $f_\cW$ in $(u,v)$-coordinates, and write $g_\cW = (g_{\cW,1},g_{\cW,2})$.
Then
\begin{equation}
g_\cX(x,y) = \begin{bmatrix} a u \\ b v \end{bmatrix},
\label{eq:gX}
\end{equation}
and write
\begin{equation}
g_\cY(u,v) = \begin{bmatrix} c_{11} u + c_{12} v + c_{13} \\ c_{21} u + c_{22} v + c_{23} \end{bmatrix},
\label{eq:gY}
\end{equation}
for some parameter-dependent coefficients $c_{11},\ldots,c_{23} \in \mathbb{R}$.

%%%%%%%%%%%%%%%%%%%%%%%%%%%%%%%%%%%%%%%%%%%%%%%%%%%%%%%%%%%%%
\begin{figure}[b!]
\begin{center}
\setlength{\unitlength}{1cm}
\begin{picture}(8,6)
\put(0,0){\includegraphics[height=6cm]{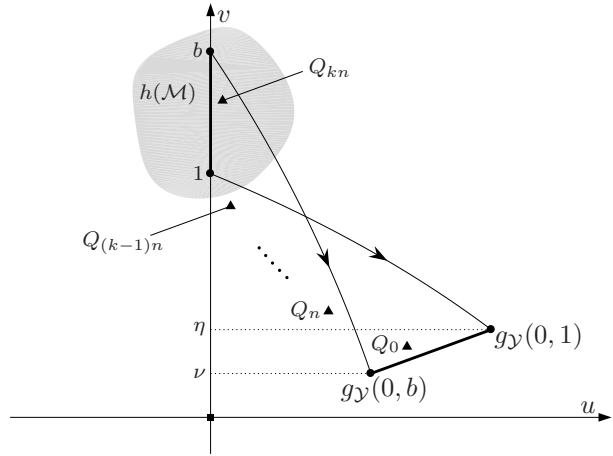}}
\put(7.58,.55){\footnotesize $u$}
\put(2.75,5.75){\footnotesize $v$}
\put(2.44,1.01){\scriptsize $\nu$}
\put(2.44,1.63){\scriptsize $\eta$}
\put(2.43,3.64){\scriptsize $1$}
\put(2.43,5.29){\scriptsize $b$}
\put(4.4,.77){\footnotesize $g_\cY(0,b)$}
\put(6.44,1.48){\footnotesize $g_\cY(0,1)$}
\put(1.72,4.72){\scriptsize $h(\cM)$}
\put(4.8,1.39){\scriptsize $Q_0$}
\put(3.73,1.84){\scriptsize $Q_n$}
\put(.96,2.77){\scriptsize $Q_{(k-1)n}$}
\put(3.97,5.07){\scriptsize $Q_{k n}$}
\end{picture}
\caption{
Phase space in $(u,v)$-coordinates, where $(u,v) = h(x,y)$, near a subsumed homoclinic connection.
Some points of an $\cX^k \cY$-cycle are shown.
For each $j \in \{ 0,\ldots,k-1 \}$, $Q_{j n}$ maps to $Q_{(j+1)n}$ under $g_\cX$, while
$Q_{k n}$ maps to $Q_0$ under $g_\cY$.
Notice $h(\cM)$ is a neighbourhood of the interval $(1,b)$ of the $v$-axis.
\label{fig:codimTwoHCCornerSchem_g}
}
\end{center}
\end{figure}
%%%%%%%%%%%%%%%%%%%%%%%%%%%%%%%%%%%%%%%%%%%%%%%%%%%%%%%%%%%%%

By the definition of $q_{\rm unstab}$, %(given just above \eqref{eq:coordinateChangeVariables})
we have $h(Z) = (0,1)$ and $h \left( f_\cX(Z) \right) = (0,b)$.
Then \eqref{eq:eta} and \eqref{eq:nu} imply
$\eta = g_{\cY,2}(0,b)$ and $\nu = g_{\cY,2}(0,1)$, see Fig.~\ref{fig:codimTwoHCCornerSchem_g}, and so
\begin{align}
c_{22} &= \frac{\eta - \nu}{b - 1},
\label{eq:c22} \\
c_{23} &= \frac{-\eta + b \nu}{b - 1}.
\label{eq:c23}
\end{align}
Formulas for the other coefficients in \eqref{eq:gY} will not be needed.

%-------------------------------------------------------------------------------
\subsection{$\cX^k \cY$-cycles}

%By Lemma \ref{le:existence} there exists a unique $\cX^k \cY$-cycle if and only if
By composing \eqref{eq:gX} and \eqref{eq:gY} we obtain
\begin{equation}
g_{\cX^k \cY}(u,v) =
\begin{bmatrix}
a^k c_{11} u + b^k c_{12} v + c_{13} \\
a^k c_{21} u + b^k c_{22} v + c_{23}
\end{bmatrix}.
\label{eq:gXkY}
\end{equation}
The trace and determinant of $\rD g_{\cX^k \cY}$ are therefore
\begin{align}
\tau_k &= a^k c_{11} + b^k c_{22} \,,
\label{eq:tauk} \\
\delta_k &= a^k b^k \left( c_{11} c_{22} - c_{12} c_{21} \right),
\label{eq:deltak}
\end{align}
respectively.
By Lemma \ref{le:existence}, there exists a unique $\cX^k \cY$-cycle
if $\det \left( I - \rD f_{\cX^k \cY} \right) \ne 0$.
This is equivalent to $\delta_k - \tau_k + 1 \ne 0$, because
$\rD f_{\cX^k \cY}$ and $\rD g_{\cX^k \cY}$ are similar
and $\det \left( I - \rD g_{\cX^k \cY} \right) = \delta_k - \tau_k + 1$.

%Let us write the $\cX^k \cY$-cycle in $(u,v)$-coordinates
%as $\left\{ Q_0,\ldots,Q_{k n + p - 1} \right\}$
For each $i$, let $Q_i = h \left( \hat{P}^{(k)}_i \right)$
(in $Q_i$ the $k$-dependence is suppressed for brevity).
Notice $Q_0$ is a fixed point of $g_{\cX^k \cY}$.
By calculating this point via \eqref{eq:gXkY} and iterating it $k$ times under $g_\cX$ we obtain
\begin{equation}
Q_{k n} = \frac{1}{\delta_k - \tau_k + 1}
\begin{bmatrix}
a^k c_{13} - a^k b^k \left( c_{22} c_{13} - c_{12} c_{23} \right) \\
b^k c_{23} + a^k b^k \left( c_{21} c_{13} - c_{11} c_{23} \right)
\end{bmatrix},
\label{eq:Qkkn}
\end{equation}
assuming $\delta_k - \tau_k + 1 \ne 0$.
We also write $Q_{k n} = \left( u^{(k)}, v^{(k)} \right)$.

%-------------------------------------------------------------------------------
\subsection{Accommodating relatively large values of $\eta$ and $\nu$}

In \eqref{eq:tauk}--\eqref{eq:Qkkn}, the largest terms are 
$b^k c_{22}$ in \eqref{eq:tauk} and $b^k c_{23}$ in \eqref{eq:Qkkn}
because $b > 1$ and $k$ is large.
But in $S_k$ we must have $|\tau_k| < 2$ (for stability)
and $v^{(k)}$ be not too large (for admissibility).
Intuitively, \eqref{eq:c22}--\eqref{eq:c23}
imply that $\eta$ and $\nu$ need to be $\cO \left( \sigma^{-k} \right)$ throughout $S_k$.
Here we establish this formally (although details are omitted for brevity).

We partition $\psi(\cN)$ into three components using some values $R_1, R_2 > 0$ as follows.
Let $\cN_1$ be all $(\eta,\nu) \in \psi(\cN)$ for which $|\eta - \nu| > R_1 \sigma^{-k}$,
let $\cN_2$ be all $(\eta,\nu) \in \psi(\cN)$ for which $|\eta - \nu| \le R_1 \sigma^{-k}$ and $|\eta| > R_2 \sigma^{-k}$,
and let $\cN_3$ be all other points in $\psi(\cN)$.
From \eqref{eq:c22} and \eqref{eq:tauk} we can find $R_1 > 0$
such that throughout $\cN_1$ we have $|\tau_k| > 2$, for all sufficiently large values of $k$,
and so, if it exists, the $\cX^k \cY$-cycle is unstable.
Such an $R_1$ can be constructed from values 
$b_{\rm min}, b_{\rm max} > 1$ for which $b_{\rm min} \le b \le b_{\rm max}$ throughout $\psi(\cN)$.
Similarly we can find $R_2 > 0$ such that throughout $\cN_2$ we have $|v^{(k)}| > b+1$, say,
for all sufficiently large values of $k$,
and so $Q_{k n} \notin h(\cM)$, assuming $\cM$ is sufficiently small.

%Let $R_1 = 3 (b_{\rm max} - 1)$.
%For any $(\eta,\nu) \in \cN$, if $|\eta - \nu| > R_1 \sigma^{-k}$,
%then, by \eqref{eq:c22} and \eqref{eq:tauk},
%$|\tau_k| > 2$, for sufficiently large values of $k$\removableFootnote{
%This actually seems like a bit of a pain to prove.
%},
%and so, if it exists, the $\cX^k \cY$-cycle is not stable (see Fig.~\ref{fig:codimTwoHCCornerSchem_c}).

%This shows that $S_k \cap \cN \subset \cN_3$,
Consequently it remains for us to consider $(\eta,\nu) \in \cN_3$
for which we may treat $\tilde{\eta} = \sigma^k \eta$
and $\tilde{\nu} = \sigma^k \nu$ as order-$1$ quantities.

%-------------------------------------------------------------------------------
\subsection{Existence}

From \eqref{eq:c22} and \eqref{eq:tauk}--\eqref{eq:deltak} we have
\begin{align*}
\tau_k &= \frac{\tilde{\eta} - \tilde{\nu}}{\sigma - 1} + \cO \left( \sigma^{-k} \right), \\
\delta_k &= \cO \left( \lambda^k \sigma^k \right),
\end{align*}
and so
\begin{equation}
\delta_k - \tau_k + 1
= 1 - \frac{\tilde{\eta} - \tilde{\nu}}{\sigma - 1}
+ \cO \left( \sigma^{-k} \right)
+ \cO \left( \lambda^k \sigma^k \right).
\label{eq:sn}
\end{equation}
By substituting $\tilde{\nu} > \tilde{\eta} - \sigma + 1 + 2 \gamma^k \sigma^k$
(which holds for all points in $\Delta_{k,\gamma}^-$)
into \eqref{eq:sn}, we obtain
%By substituting $\eta < \sigma^{-(k-1)} - \gamma^k$
%and $\nu > \sigma^{-k} + \gamma^k$ into \eqref{eq:denom},
%we find that above and left of the bottom-right vertex of $\Delta_{k,\gamma}^-$ we have
$\delta_k - \tau_k + 1 > \frac{2 \sigma^k \gamma^k}{\sigma - 1}
+ \cO \left( \sigma^{-k} \right)
+ \cO \left( \lambda^k \sigma^k \right)$,
which is positive for sufficiently large values of $k$.
Thus the $\cX^k \cY$-cycle exists and is unique in $\Delta_{k,\gamma}^-$ (for large $k$).

Similarly by substituting $\tilde{\nu} < \tilde{\eta} - \sigma + 1 - 2 \gamma^k \sigma^k$ into \eqref{eq:sn}
we obtain $\delta_k - \tau_k + 1 < 0$, for large $k$,
and so here the $\cX^k \cY$-cycle cannot be stable.

%-------------------------------------------------------------------------------
\subsection{Stability}

By substituting $\tilde{\nu} < \tilde{\eta} + \sigma - 1 - \gamma^k \sigma^k$
(i.e.~below the angled edge of $\Delta_{k,\gamma}^-$) into
\begin{equation}
\delta_k + \tau_k + 1
= 1 + \frac{\tilde{\eta} - \tilde{\nu}}{\sigma - 1}
+ \cO \left( \sigma^{-k} \right)
+ \cO \left( \lambda^k \sigma^k \right),
\label{eq:pd}
\end{equation}
we obtain
$\delta_k + \tau_k + 1 > \frac{\sigma^k \gamma^k}{\sigma - 1}
+ \cO \left( \sigma^{-k} \right)
+ \cO \left( \lambda^k \sigma^k \right)$,
which is positive for sufficiently large values of $k$.
Since also $\delta_k \to 0$ as $k \to \infty$,
we conclude that in $\Delta_{k,\gamma}^-$ the pair $(\tau_k,\delta_k)$ belongs to the
triangle of Fig.~\ref{fig:codimTwoHCCornerSchem_c}
and so if the $\cX^k \cY$-cycle is admissible then it is also stable.

By similarly substituting $\tilde{\nu} > \tilde{\eta} + \sigma - 1 + \gamma^k \sigma^k$ into \eqref{eq:pd}
we obtain $\delta_k + \tau_k + 1 < 0$, for large $k$,
and so here the $\cX^k \cY$-cycle cannot be stable.

%-------------------------------------------------------------------------------
\subsection{Admissibility}

From \eqref{eq:c23} and \eqref{eq:Qkkn} we have $u^{(k)} = \cO \left( \lambda^k \right)$ and
\begin{equation}
v^{(k)} = v^{(\infty)}
+ \cO \left( \sigma^{-k} \right)
+ \cO \left( \lambda^k \sigma^k \right),
\label{eq:vk}
\end{equation}
where
\begin{equation}
v^{(\infty)} = \frac{-\tilde{\eta} + \sigma \tilde{\nu}}{\sigma - 1 - \tilde{\eta} + \tilde{\nu}}.
\label{eq:vinfty}
\end{equation}
If $\tilde{\eta}$ and $\tilde{\nu}$ are such that $1 < v^{(\infty)} < \sigma$,
then Assumption \ref{as:main} implies $Q_{k n} \in \Omega_\cY$
(also $Q_{k n} \in h(\cM)$) for large values of $k$.
Also for each $j = 0,\ldots,k-1$ the point $Q_{j n}$
is in $\cB$ or close to $\cL(P_0,Z)$, for large values of $k$,
and so $Q_{j n} \in \Omega_\cX$.
In this case the $\cX^k \cY$-cycle is admissible and $(\eta,\nu) \in \psi(S_k)$.

If instead $v^{(\infty)} < 1$ then for sufficiently large values of $k$ either
$Q_{k n} \in \Omega_\cX$
(in which case the $\cX^k \cY$-cycle is virtual because $\cX_0 \ne \cY_0$)
or $Q_{k n} \notin h(\cM)$.
In either case $(\eta,\nu) \notin \psi(S_k)$.
If $v^{(\infty)} > \sigma$, then either $Q_{(k-1)n} \in \Omega_\cY$
or $Q_{k n} \notin h(\cM)$ and we reach the same conclusion.

Simply solving $v^{(\infty)} = 1$ and $v^{(\infty)} = \sigma$ yields
$\tilde{\nu} = 1$ and $\tilde{\eta} = \sigma$, respectively,
and this provides a rough explanation for the 
vertical and horizontal edges of the $\Delta_{k,\gamma}^\pm$.
More formally it is a simple exercise to use the formula \eqref{eq:vinfty}
to show that in $\Delta_{k,\gamma}^-$ we have
\begin{equation}
1 + \tfrac{1}{2} \halfOfACommaSpace \beta^k \sigma^k
+ \cO \left( \sigma^{-k} \right)
+ \cO \left( \lambda^k \sigma^k \right) <
v^{(\infty)} < \sigma - \tfrac{1}{2} \halfOfACommaSpace \beta^k \sigma^k
+ \cO \left( \sigma^{-k} \right)
+ \cO \left( \lambda^k \sigma^k \right),
\nonumber
\end{equation}
and so in $\Delta_{k,\gamma}^-$ the $\cX^k \cY$ is admissible
with $Q_{k n} \in h(\cM)$ for sufficiently large values of $k$.
Also if $\tilde{\nu} \ge \tilde{\eta} - \sigma + 1 - 2 \gamma^k \sigma^k$
and $\tilde{\nu} \le \tilde{\eta} + \sigma - 1 + \gamma^k \sigma^k$
(above it was shown that these are necessary for the $\cX^k \cY$-cycle to be stable),
then outside $\Delta_{k,\gamma}^+$ we have either
$v^{(\infty)} < 1 - \tfrac{1}{2} \halfOfACommaSpace \beta^k \sigma^k
+ \cO \left( \sigma^{-k} \right)
+ \cO \left( \lambda^k \sigma^k \right)$ or
$v^{(\infty)} > \sigma + \tfrac{1}{2} \halfOfACommaSpace \beta^k \sigma^k
+ \cO \left( \sigma^{-k} \right)
+ \cO \left( \lambda^k \sigma^k \right)$
and so the $\cX^k \cY$-cycle is either virtual or $Q_{k n} \notin h(\cM)$.
\manualEndProof

%===============================================================================
\section{Examples}
\label{sec:examples}
\setcounter{equation}{0}

%-------------------------------------------------------------------------------
\subsection{A discontinuous map}

Here we consider the discontinuous map of \cite{Mi13}:
\begin{equation}
f(x,y) = \begin{cases}
f_1(x,y), & x < 0, \\
f_2(x,y), & x \ge 0,
\end{cases}
\label{eq:fMi13}
\end{equation}
where
\begin{align}
f_1(x,y) &= \begin{bmatrix} \bar{\lambda}_1 x + y + \bar{a} \\ \bar{\beta} x \end{bmatrix}, \\
f_2(x,y) &= \begin{bmatrix} \bar{\lambda}_2 x + y + \bar{b} \\ \bar{\beta} x \end{bmatrix}.
\nonumber
\end{align}
Bars have been added to the parameters
to avoid confusion with the notation that has already been developed.
As in Figure 7 of \cite{Mi13} we fix
\begin{align}
\bar{b} &= -2, &
\bar{\lambda}_2 &= 1.5, &
\bar{\beta} &= 0.3,
\label{eq:ParamExDisc}
\end{align}
and vary the remaining two parameters.

Fig.~\ref{fig:codimTwoHCCornerTonguesExDisc} shows a two-parameter bifurcation diagram.
Regions where there exists a stable periodic solution (up to period $30$) are coloured by period.
The two white lines are curves of codimension-one homoclinic corners.
The intersection of these curves is a codimension-two point at which
a fixed point $P_0$ in $x > 0$ has
eigenvalues $\lambda \approx -0.18$ and $\sigma \approx 1.68$
and a subsumed homoclinic connection.
Fig.~\ref{fig:codimTwoHCCornerQQExDisc} shows a phase portrait
at sample parameter values near the codimension-two point.
%Here there exist stable $2^5 1$ and $2^6 1$-cycles.

The codimension-two point satisfies the conditions of Theorem \ref{th:main}.
Here $\cX = 2$ and $\cY = 1$.
The point $Z$ lies at a transverse intersection between
$E^u(P_0)$ and the switching manifold $x=0$,
and condition (iii) of Assumption \ref{as:main} can be verified by inspection.
Also $\det(J) \ne 0$ is a consequence of the transversal intersection
between the curves of homoclinic corners.

The regions $S_k$ predicted by Theorem \ref{th:main}
are labelled in Fig.~\ref{fig:codimTwoHCCornerTonguesExDisc} up to $k = 6$
and visible for some larger values of $k$
(other coloured regions correspond to multi-round periodic solutions).
The phase portrait in Fig.~\ref{fig:codimTwoHCCornerQQExDisc} corresponds to
parameter values in $S_5 \cap S_6$
and the corresponding stable $2^5 1$ and $2^6 1$-cycles are plotted.

%%%%%%%%%%%%%%%%%%%%%%%%%%%%%%%%%%%%%%%%%%%%%%%%%%%%%%%%%%%%%
\begin{figure}[b!]
\begin{center}
\setlength{\unitlength}{1cm}
\begin{picture}(12,6)
\put(0,0){\includegraphics[height=6cm]{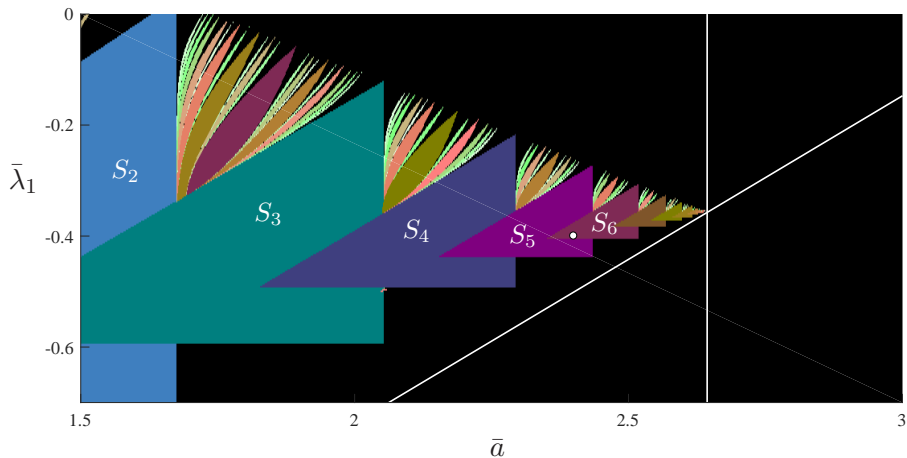}}
\put(6.4,0){\small $\bar{a}$}
\put(0,3.56){\small $\bar{\lambda}_1$}
\put(1.37,3.7){\footnotesize \color{white} $S_2$}
\put(3.27,3.13){\footnotesize \color{white} $S_3$}
\put(5.24,2.9){\footnotesize \color{white} $S_4$}
\put(6.65,2.84){\footnotesize \color{white} $S_5$}
\put(7.74,3){\footnotesize \color{white} $S_6$}
\end{picture}
\caption{
Regions where the discontinuous map
\eqref{eq:fMi13} with \eqref{eq:ParamExDisc}
has a stable periodic solution (different colours correspond to different periods).
This was computed numerically via a brute-force search on a $1024 \times 256$ grid of
$\left( \bar{a},\bar{\lambda}_1 \right)$-values up to period $30$.
At points where multiple stable periodic solutions were found the highest period is indicated.
The white lines are curves of homoclinic corners that intersect at a subsumed homoclinic connection.
The white circle indicates the parameter values of Fig.~\ref{fig:codimTwoHCCornerQQExDisc}.
\label{fig:codimTwoHCCornerTonguesExDisc}
}
\end{center}
\end{figure}
%%%%%%%%%%%%%%%%%%%%%%%%%%%%%%%%%%%%%%%%%%%%%%%%%%%%%%%%%%%%%

%%%%%%%%%%%%%%%%%%%%%%%%%%%%%%%%%%%%%%%%%%%%%%%%%%%%%%%%%%%%%
\begin{figure}[t!]
\begin{center}
\setlength{\unitlength}{1cm}
\begin{picture}(8,6)
\put(0,0){\includegraphics[height=6cm]{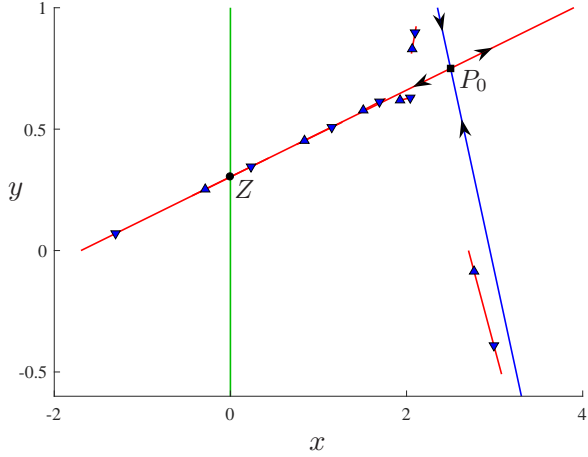}}
\put(4,0){\small $x$}
\put(0,3.39){\small $y$}
\put(6,4.82){\footnotesize $P_0$}
\put(3.01,3.35){\footnotesize $Z$}
\end{picture}
\caption{
A phase portrait of the discontinuous map
\eqref{eq:fMi13} with \eqref{eq:ParamExDisc} and $\left( \bar{a}, \bar{\lambda}_1 \right) = (2.4,-0.4)$.
Parts of the stable and unstable manifolds of the fixed point $P_0$ are
coloured blue and red, respectively.
The intersection $Z$ (defined in \S\ref{sec:thm}) of one branch of $W^u(P_0)$ with $\Sigma$ ($x=0$) is indicated.
The stable $2^5 1$ and $2^6 1$-cycles are shown with two sets of triangles of different orientations.
\label{fig:codimTwoHCCornerQQExDisc}
}
\end{center}
\end{figure}
%%%%%%%%%%%%%%%%%%%%%%%%%%%%%%%%%%%%%%%%%%%%%%%%%%%%%%%%%%%%%

%-------------------------------------------------------------------------------
\subsection{The border-collision normal form}

The continuous map
\begin{equation}
f(x,y) = \begin{cases}
f_1(x,y), & x \le 0, \\
f_2(x,y), & x \ge 0,
\end{cases}
\label{eq:bcnf}
\end{equation}
where
\begin{align}
f_1(x,y) &= \begin{bmatrix} \tau_1 x + y + 1 \\ -\delta_1 x \end{bmatrix}, \\
f_2(x,y) &= \begin{bmatrix} \tau_2 x + y + 1 \\ -\delta_2 x \end{bmatrix}.
\nonumber
\end{align}
is known as the two-dimensional border-collision normal form
(except the border-collision bifurcation parameter, usually called $\mu$,
has been scaled to $1$) \cite{NuYo92,Si16}.
Fig.~\ref{fig:codimTwoHCCornerTonguesExCont} shows a two-parameter bifurcation diagram using
\begin{align}
\tau_1 &= 2, &
\delta_1 &= 0.75.
\label{eq:ParamExCont}
\end{align}
%as in \cite{Si16b}.
Four distinct curves of homoclinic corners emanate from a common intersection point at $(\delta_2,\tau_2) = (1.5,-0.5)$.
Here a fixed point $P_0$ in $x < 0$
has eigenvalues $\lambda = 0.5$ and $\sigma = 1.5$ and a subsumed homoclinic connection.
The conditions of Theorem \ref{th:main} can be readily verified
with $\cX = 1$ and $\cY = 22$,
and the regions $S_k$ are shown in Fig.~\ref{fig:codimTwoHCCornerTonguesExCont}.
Fig.~\ref{fig:codimTwoHCCornerQQExCont} shows a phase portrait at parameter values
relatively near the codimension-two point.

%%%%%%%%%%%%%%%%%%%%%%%%%%%%%%%%%%%%%%%%%%%%%%%%%%%%%%%%%%%%%
\begin{figure}[t!]
\begin{center}
\setlength{\unitlength}{1cm}
\begin{picture}(12,6)
\put(0,0){\includegraphics[height=6cm]{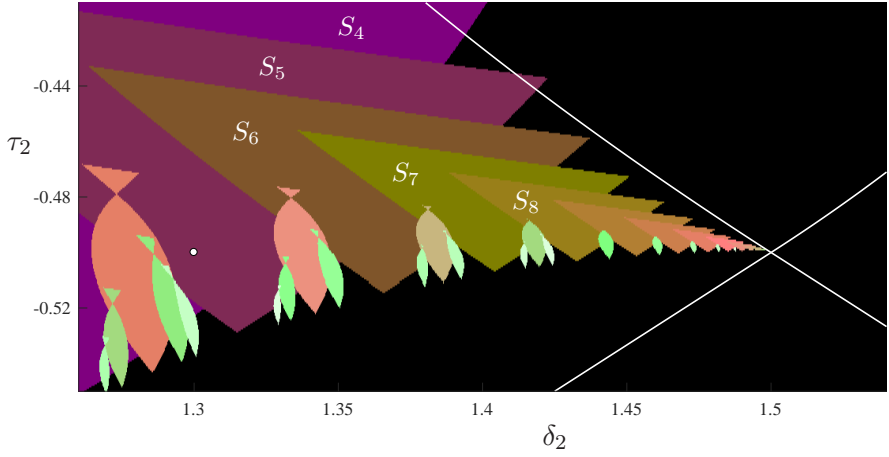}}
\put(7.12,0){\small $\delta_2$}
\put(0,3.93){\small $\tau_2$}
\put(4.4,5.45){\footnotesize \color{white} $S_4$}
\put(3.35,4.92){\footnotesize \color{white} $S_5$}
\put(3,4.06){\footnotesize \color{white} $S_6$}
\put(5.12,3.51){\footnotesize \color{white} $S_7$}
\put(6.72,3.14){\footnotesize \color{white} $S_8$}
\end{picture}
\caption{
A two-parameter bifurcation diagram of \eqref{eq:bcnf} with \eqref{eq:ParamExCont}
using the same conventions as Fig.~\ref{fig:codimTwoHCCornerTonguesExDisc}.
The white circle indicates the parameter values of Fig.~\ref{fig:codimTwoHCCornerQQExCont}.
\label{fig:codimTwoHCCornerTonguesExCont}
}
\end{center}
\end{figure}
%%%%%%%%%%%%%%%%%%%%%%%%%%%%%%%%%%%%%%%%%%%%%%%%%%%%%%%%%%%%%

%%%%%%%%%%%%%%%%%%%%%%%%%%%%%%%%%%%%%%%%%%%%%%%%%%%%%%%%%%%%%
\begin{figure}[t!]
\begin{center}
\setlength{\unitlength}{1cm}
\begin{picture}(8,6)
\put(0,0){\includegraphics[height=6cm]{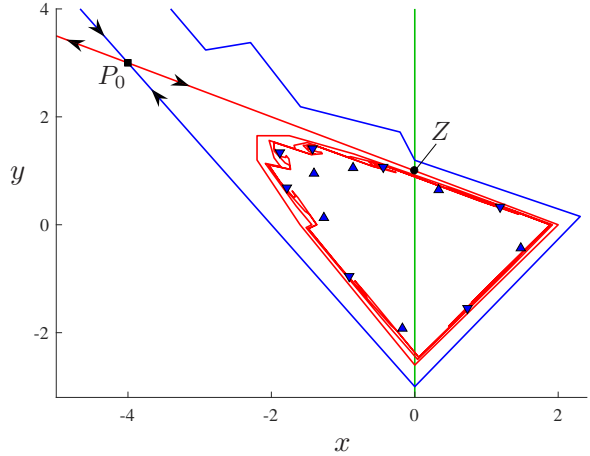}}
\put(4.31,0){\small $x$}
\put(0,3.62){\small $y$}
\put(1.17,4.88){\footnotesize $P_0$}
\put(5.59,4.14){\footnotesize $Z$}
\end{picture}
\caption{
A phase portrait of \eqref{eq:bcnf} with \eqref{eq:ParamExCont} and $(\delta_2,\tau_2) = (1.3,-0.5)$
using the same conventions as Fig.~\ref{fig:codimTwoHCCornerQQExDisc}.
The stable periodic solutions (shown with triangles of different orientations) are $1^4 2^2$ and $1^5 2^2$-cycles.
\label{fig:codimTwoHCCornerQQExCont}
}
\end{center}
\end{figure}
%%%%%%%%%%%%%%%%%%%%%%%%%%%%%%%%%%%%%%%%%%%%%%%%%%%%%%%%%%%%%

%-------------------------------------------------------------------------------
\subsection{An example with $n > 1$}

Finally we provide an example for which the $\cX$-cycle is not a fixed point.
Fig.~\ref{fig:codimTwoHCCornerTonguesExContPeriod3}
shows a different two-parameter slice of the four-dimensional parameter space
of the border-collision normal form \eqref{eq:bcnf}.
Specifically we fix
\begin{align}
\tau_2 &= -2.5, &
\delta_2 &= 2.
\label{eq:ParamExContPeriod3}
\end{align}
Four curves of homoclinic corners intersect at
$(\tau_1,\delta_1) = \left( -\frac{23}{33}, \frac{13}{66} \right)$.
At this point \eqref{eq:bcnf} has an $\cX$-cycle, where $\cX = 212$,
with eigenvalues $\lambda = \frac{4}{11}$ and $\sigma = \frac{13}{6}$
and a subsumed homoclinic connection.
Fig.~\ref{fig:codimTwoHCCornerQQExContPeriod3} shows a phase portrait
at nearby parameter values.
Here the conditions of Theorem \ref{th:main} are satisfied
using $\cY = 12$ (the cyclic permutation of $\cX$ has been chosen so that $\cX_0 = \cY_0$).
Again the regions $S_k$
are in accordance with Theorem \ref{th:main}, at least for sufficiently large values of $k$
(the region $S_4$ has an extra component corresponding to stable $(212)^4 22$-cycles).

%%%%%%%%%%%%%%%%%%%%%%%%%%%%%%%%%%%%%%%%%%%%%%%%%%%%%%%%%%%%%
\begin{figure}[t!]
\begin{center}
\setlength{\unitlength}{1cm}
\begin{picture}(12,6)
\put(0,0){\includegraphics[height=6cm]{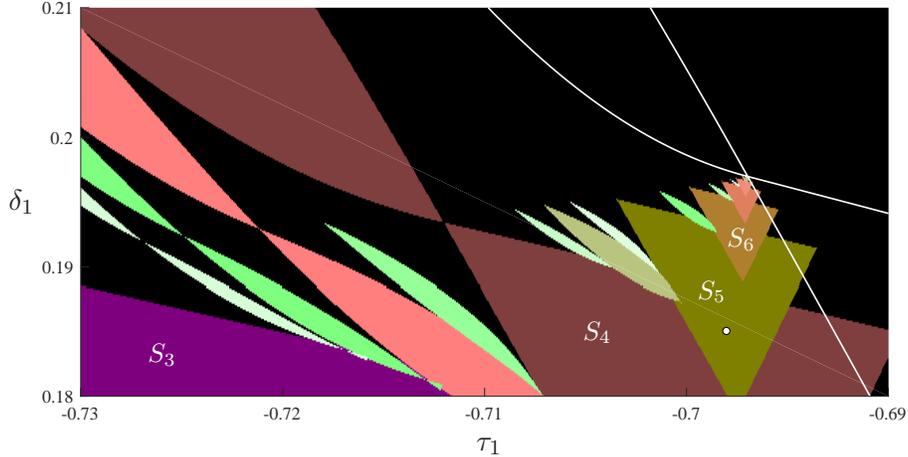}}
\put(6.22,0){\small $\tau_1$}
\put(0,3.2){\small $\delta_1$}
\put(1.85,1.17){\footnotesize \color{white} $S_3$}
\put(7.64,1.47){\footnotesize \color{white} $S_4$}
\put(9.16,2.01){\footnotesize \color{white} $S_5$}
\put(9.55,2.71){\footnotesize \color{white} $S_6$}
\end{picture}
\caption{
A two-parameter bifurcation diagram of \eqref{eq:bcnf} with \eqref{eq:ParamExContPeriod3}
using the same conventions as Fig.~\ref{fig:codimTwoHCCornerTonguesExDisc},
except periodic solutions have been computed up to period $50$.
The white circle indicates the parameter values of Fig.~\ref{fig:codimTwoHCCornerQQExContPeriod3}.
\label{fig:codimTwoHCCornerTonguesExContPeriod3}
}
\end{center}
\end{figure}
%%%%%%%%%%%%%%%%%%%%%%%%%%%%%%%%%%%%%%%%%%%%%%%%%%%%%%%%%%%%%

%%%%%%%%%%%%%%%%%%%%%%%%%%%%%%%%%%%%%%%%%%%%%%%%%%%%%%%%%%%%%
\begin{figure}[t!]
\begin{center}
\setlength{\unitlength}{1cm}
\begin{picture}(8,6)
\put(0,0){\includegraphics[height=6cm]{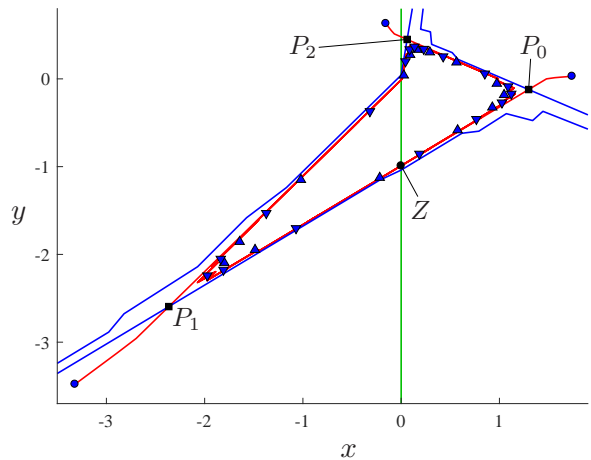}}
\put(4.38,0){\small $x$}
\put(0,3.17){\small $y$}
\put(6.76,5.37){\footnotesize $P_0$}
\put(2.14,1.75){\footnotesize $P_1$}
\put(3.69,5.36){\footnotesize $P_2$}
\put(5.31,3.17){\footnotesize $Z$}
\end{picture}
\caption{
A phase portrait of \eqref{eq:bcnf} with \eqref{eq:ParamExContPeriod3} and $(\tau_1,\delta_1) = (-0.698,0.185)$	
using the same conventions as Fig.~\ref{fig:codimTwoHCCornerQQExDisc}.
Stable $\cX^4 \cY$ and $\cX^5 \cY$-cycles are shown with triangles of different orientations.
There also exists a stable $211$-cycle shown with circles.
\label{fig:codimTwoHCCornerQQExContPeriod3}
}
\end{center}
\end{figure}
%%%%%%%%%%%%%%%%%%%%%%%%%%%%%%%%%%%%%%%%%%%%%%%%%%%%%%%%%%%%%

%===============================================================================
\section{Discussion}
\label{sec:conc}
\setcounter{equation}{0}

Stable $\cX^k \cY$-cycles do not exist near codimension-one homoclinic corners \cite{Si16b}.
This paper shows that near certain codimension-two homoclinic connections,
stable $\cX^k \cY$-cycles do exist in nearby regions $S_k$, Theorem \ref{th:main}.
To leading order $S_k$ is the set $\psi^{-1} \left( \Delta_{k,0}^+ \right)$
where $\Delta_{k,0}^+$ is triangular.
This explains, quantitatively, the bulk of the two-parameter bifurcation diagrams shown in
Figs.~\ref{fig:codimTwoHCCornerTonguesExDisc},
\ref{fig:codimTwoHCCornerTonguesExCont}, and \ref{fig:codimTwoHCCornerTonguesExContPeriod3}
for three different examples.
Each $S_k$ overlaps only $S_{k-1}$ and $S_{k+1}$ for large $k$, Proposition \ref{pr:overlap}.
In such overlaps there exist two stable periodic solutions,
as shown in Figs.~\ref{fig:codimTwoHCCornerQQExDisc}, \ref{fig:codimTwoHCCornerQQExCont},
and \ref{fig:codimTwoHCCornerQQExContPeriod3}.

Near the codimension-three scenario that we also have $\lambda \sigma = 1$,
for any $N \ge 1$ there exist open regions of parameter space
in which $N$ different regions $S_k$ overlap \cite{DoLa08,Si14b,Si14}.
For smooth area-preserving maps the analogous scenario is only codimension-two
because the condition $\lambda \sigma = 1$ is satisfied automatically \cite{GoGo09,GoSh05}.

It remains to generalise Theorem \ref{th:main} to higher dimensions, following \cite{SiTu17}.
It also remains to characterise multi-round periodic solutions.
Certainly from the above examples
regions where multi-round periodic solutions are admissible and stable
appear to exhibit a consistent structure nestled between the $S_k$.

%===============================================================================
\section{Acknowledgements}
\label{sec:ack}
\setcounter{equation}{0}

The author gratefully acknowledges the assistance of Edward (Zhiyang) Chen
regarding calculations of the boundaries of the $S_k$.

%{\footnotesize
%\bibliographystyle{plain}
%\bibliography{../DynSyst,../MathBio,../Misc,../OtherTheory,../PWS,../Stoch}
%}

\end{document}